\RequirePackage{fix-cm}
\documentclass{svjour3}                     
\smartqed  
\usepackage{graphicx}
\usepackage{amsmath}
\usepackage{graphicx}
\usepackage{hyperref}
\usepackage{amsfonts}
\usepackage{amssymb}
\usepackage{enumerate}

\DeclareOldFontCommand{\bf}{\normalfont\bfseries}{\mathbf}
\DeclareOldFontCommand{\cal}{\normalfont\bfseries}{\mathcal}
\DeclareMathOperator\CL{CL}
\DeclareMathOperator*\argmax{arg\,max}
\DeclareMathOperator\Li{Li}
\DeclareMathOperator\Ls{Ls}

\newtheorem{assumption}[theorem]{Assumption}

\def\RR{{\mathbb R}}

\def\B{{\mathcal B}}

\def\X{{\mathcal X}}

\begin{document}

\title{Robust Contracting in General Contract Spaces\thanks{Horst and Beissner gratefully acknowledge financial support by the CRC TRR 190 \sl{Rationality and competition - the economic performance of individuals and firms}
}
}


\author{ Julio Backhoff-Veraguas   \qquad     \and \qquad \newline
        Patrick Beissner \and Ulrich Horst
}

\authorrunning{Julio Backhoff-Veraguas, Patrick Beissner, Ulrich Horst} 

\institute{Julio Backhoff-Veraguas  \at
              University of Vienna, Faculty of Mathematics. Austria.\\
              \email{julio.backhoff@univie.ac.at}            \\
           \and
           Patrick Beissner \at
              The Australian National University, Research School of Economics. Australia.\\
              \email{patrick.beissner@anu.edu.au}           \\ 
              \and
           Ulrich Horst \at
              Humboldt-Universit\"at zu Berlin, Department of Mathematics and School of Business and Economics. Germany. \\
              \email{horst@math.hu-berlin.de}            \\
}

\date{\today}

\maketitle

\begin{abstract}
We consider a general framework of optimal mechanism design under {adverse selection} and ambiguity about the type distribution of agents. We prove the existence of optimal mechanisms under minimal assumptions on the contract space and prove that centralized contracting implemented via mechanisms  is equivalent to delegated contracting implemented via a contract menu under these assumptions. Our abstract existence results are applied to a series of applications that include models of optimal risk sharing and of optimal portfolio delegation. 
%
%

\end{abstract}

\section{Introduction}

We  consider a model of optimal mechanism design under adverse selection under minimal assumptions on the contract space. We allow the agent's preferences to be type-dependent (private information of the agent) and the principal to hold ambiguous beliefs about the distribution of these types. The principal's preferences for ambiguity  are captured by the maxmin expected utility model of \cite{gilboa1989maxmin} or, more generally,  by variational preferences  \`a la \cite{maccheroni2006ambiguity}. 
Our main result establishes the existence of optimal contracts under minimal assumptions on the contract space and the set of beliefs. Importantly, we do not require the contract space to be metrizable. Nonmetrizable contract spaces arise naturally in principal-agent models when  the utility levels of the agent are the relevant contract variables. The contracting problem over the original contracting space, and the one over the space of the resulting utility levels, are equivalent under mild conditions on the agent's utility function.  Quantifying  the  decision variable of the principal accounted in utility units of the agent is a well-known approach; see \cite{spear1987repeated}, \cite{schattler1993first} and more recently  \cite{sannikov2008continuous}. In our case, the transformation has a  different purpose. After the transformation, the agent's utility, as  a function of type and contract, reduces to a bilinear form on the contract/type space pair. This significantly simplifies the incentive compatibility constraint. Moreover, conditions on the continuity of the agent's utility boil down to the joint continuity of this bilinear form. As a result, a solvable model of optimal contracts with  large type (contract) space typically only allows for a small contract (type) space. This dichotomy is well known for the  commodity-price duality in general equilibrium analysis. We  show in a series of applications that this seemingly mathematical generalization is of economic relevance. Our applications include models of optimal reinsurance and market optimized risk sharing, and models of optimal portfolio delegation. The latter models combine adverse selection and moral hazard effects. 

 {Our  abstract existence result  extends \cite{Page1992}\footnote{In this study,  a  characterization of incentive compatibility is established in a  Polish  type space  combined with a compact metric contract space. 
} by dropping the assumption of a metrizable contract space. {The central idea of this line of reasoning is to consider contract menus (i.e.\ closed subsets of contracts) rather than contract mechanisms (i.e.\ functions). However,} {without a  metric on the contract space, we must substitute the usual Hausdorff metric topology on the set of contract menus  by the {\em Fell topology} (the topology of closed convergence) under which the hyperspace of all closed subsets is compact, as long as  the contract space is compact.\footnote{Appendix \ref{abstract result} provides a detailed account of all involved topological concepts.} Working with the space of closed subsets under  the Fell topology,  we can show continuity of the utility functions defined on sets  and hence the existence of an optimal contract menu. 
Subsequently, we prove that delegated contracting implemented via  contract menus is equivalent to centralized contracting implemented via contract mechanisms. This equivalence is not an obvious result, because the usual measurable selection theorems on which the proof of the revelation principle is based, do not apply in nonmetrizable spaces.}


{
Our framework is flexible enough to allow the principal to hold uncertain beliefs about the type of the agent. As such this paper also contributes to the growing literature on robust design of incentives. Following \cite{carroll2019robustness} this literature can be roughly decomposed as follows. First, ``robustness to technology or preferences'' where a principal does not exactly know the agent's possible actions and/or the distribution over relevant output variables; see \cite{carroll2015robustness}, \cite{frankel2014aligned}, \cite{garrett2014robustness} for more details. Second, ``robustness to beliefs and strategic behavior'', where agents hold ambiguous beliefs about each other. For instance, \cite{carroll2018information} considers a planner designing a mechanism for trade who is unsure about the beliefs that other players have about each other, and whose objective is worst-case expected welfare. \cite{BrooksDu2020}, following up on \cite{Du2018} consider robustness to beliefs in common-value auctions where buyers observe noisy information about the value of the good and the auctioneer is uncertain about what the buyers know. Third, ``robustness to distribution'' where a principal has only partial information about the distribution over agents' preference types.\footnote{For completeness, other forms of robustness are related to collusion, renegotiation, and robustness of standard mechanisms.} For instance, \cite{bergemann2011robust} and \cite{auster2018robust} consider buyer-seller models with private, respectively common values where the buyer (principal) is confronted with ambiguity about the quality distribution  of a good in question. 

The present paper focuses on robustness to preference types with a special emphasis on applications to settings in which types represent the agent's belief about the states of the world. The main goal of the paper is to identify minimal assumptions on the contract space that still guarantee the existence of optimal contracts and to understand the relation between contract and type space in such settings. This level of generality clearly prevents closed form solutions for optimal contracts; the impact of robustness on the specific structure of optimal contracts is much better analyzed on a case-by-case basis and is well beyond the scope of this paper.  }

The rest of the paper is organized as follows. Section \ref{sec:motiv_examples} provides several examples that account for robustness and a contract space that cannot be captured by existing results. Section \ref{abstract result menu} provides the results on the existence of an optimal contract.  Section \ref{sec:applications} provides a series of applications. Section \ref{concl} concludes  by {comparing the results of the present work with those in the literature}.  Appendix A recalls a series of abstract topological results; Appendix B presents the proofs.


\section{Motivating Examples}\label{sec:motiv_examples}

In this section, we present three motivating examples that illustrate the way in which nonmetrizable {\sl compact} contract spaces over {\sl utility units} arise naturally in models of reinsurance and optimal risk sharing with {\sl non-compact} metric contract spaces. We analyze each example in greater detail in Section \ref{sec:applications}. 

{ In what follows we denote by $L^0(Q), L^1(Q)$, and $L^\infty(Q)$ the class of all almost surely finite, integrable, respectively essentially bounded random variables defined on some probability space $(\Omega, {\cal F}, Q)$. On $ L^0(Q)$ we consider the topology of convergence in $Q$-probability; the set $L^p(Q)$, {$1 \leq p \leq \infty$}, is equipped with the usual $L^p$-norm $\|\cdot\|_p$. A set of random variables $X$ is called uniformly norm bounded if $\sup_{x \in X}\|x\|_\infty < \infty$}. It is called bounded in $Q$-probability if  $\lim_{R\nearrow\infty}\sup_{x\in X}Q(|x|\geq R)=0$. 

\subsection{A reinsurance model with utilities on the positive half line}
\label{sec:reinsurance_halfline}

Let us consider a reinsurance model between an uninformed principal and an informed agent. Risk exchange can occur over a set ${\cal X}$ of random variables defined on a probability space $(\Omega,{\cal F}, Q)$. The principal's  endowment is given by a random variable $e_p \in L^1(Q)$. She may hedge her risk by exchanging payoffs $x \in {\cal X}$ with the agent. Her utility from a risk transfer $x$ is given by  
\[
        E_{P'}[v(e_p-x)]
\]
for some bounded, concave utility function $v:\mathbb R_+ \to \mathbb{R}$ on the positive half line, and some belief $P'$ about the distribution of the states of the world. We assume that the belief is equivalent to $Q$ with bounded density.

 The agent is endowed with a claim $e_a \in L^1(Q)$. His utility from a risk transfer $x \in {\cal X}$ is given by 
$$E_{P}\left[u(e_a+x)\right] $$
for some strictly increasing, concave, continuous utility function $u:\mathbb R_+ \to \mathbb R$ on the positive half line and some belief $P$ that is also equivalent to $Q$ with bounded density. We refer to $P$ as the agent's {type}. 
Types are private knowledge to the agent; the principal only knows the set of possible types $\mathcal{Q}$. We identify  ${\cal Q}$ with the set of densities $\frac{\textnormal{d}P}{\textnormal{d}Q}$. 
{Incorporating } the non-negativity constraint on both parties' payoff, the set of admissible transfers (the contract space {in payoff units}) is given by 
\[
	X :=\left \{  x \in {\cal X} : -e_a \leq x \leq e_p  \right\} .
\]

\subsubsection{The contracting problem over contingent claims}

The principal faces the problem of designing an optimal risk-sharing rule $x: {\cal Q} \to X$ that assigns a transfer to any agent type $P \in {\cal Q}$ and that maximizes her utility, subject to the usual individual rationality and incentive compatibility {constraint}. We assume that the principal's utility from a risk-sharing mechanism $x: {\cal Q} \to X$ is given by a variational utility function ({see \cite{maccheroni2006ambiguity}}) of the form
\begin{equation} \label{problem1}
	 \inf_{\kappa \in \mathbb{K}}  \left\{ \int_{{\cal Q}}  E_{P'}\left[ v\left(e_p-x_P\right) \right] \textnormal{d}\kappa(P) + \alpha(\kappa) \right\}
\end{equation}
where $\mathbb{K} { \subset {\Delta (\cal Q)}}$ is a set of probability measures on ${\cal Q}$, and $\alpha: \mathbb K \to \mathbb R$ is a convex penalty function. Our choice of utility function allows the principal to be uncertain about the distribution of the agent type. The principal's optimization problem is thus given by: find a mechanism $x: {\cal Q} \to {X}$ that maximizes
   \begin{equation}\label{problem1}
  \begin{split}
 & \quad \inf_{\kappa \in \mathbb{K}}  \left\{ \int_{{\cal Q}}  E_{P'}\left[ v\left(e_p{ -}x_P\right) \right]\textnormal{d} \kappa(P) + \alpha(\kappa) \right\} \\ 
 \text{subject to} 
 & \quad {P\mapsto  E_{P'}\left[ v\left(e_p{ -}x_P\right) \right]\qquad \qquad\:\: \quad ~ \text{is measurable}, }
\\
& \quad E_{P}[u(x_P{ +e_a}) -u(e_a)]\geq 0,   \quad\qquad P \in {\cal Q} 
\\
&\quad E_{P}[u(x_P{ +e_a})-u(x_{\hat{P}}{ +e_a})]\geq 0,   \quad P,\hat P \in {\cal Q}, 
 \end{split}
 \end{equation}
 { where the last two constraints denote the individual rationality (i.e.\ participation) and truth-telling (i.e.\ incentive compatibility) constraints, respectively. The first of these guarantees that the mechanism is better than ``offering nothing'' and the second one makes sure that the agent is better-off if he reports his true type rather than lie.}
The set $X$ is $\sigma(L^1(Q),L^\infty(Q))$-compact.\footnote{{This follows from the Dunford-Pettis theorem, 
 see \cite[Theorem A.45]{FS}}.}
However, {it} is {usually not norm compact.\footnote{For the lack of norm compactness, observe as an illustration that when $-e_a<e_p$ are both constants, $\Omega=[0,1]$, and $Q$ is Lebesgue measure on $[0,1]$, then $X$ contains a sequence of $\{-e_a,e_p\}$-valued functions that oscillate increasingly, thus having no accumulation point w.r.t.\ $Q$-a.s.\ convergence and hence neither in $L^1(Q)$-norm.  By the Riesz lemma, the unit ball $\{x: \|x\| \leq 1\}$ in a normed vector space $(X,\|\cdot\|)$ is compact if and only if $X$ is finite dimensional.}} At the same time, the agent's utility function is continuous w.r.t. the norm topology under mild technical conditions but usually fails to be continuous w.r.t. the $\sigma(L^1(Q),L^\infty(Q))$-topology (unless the agent is risk neutral).\footnote{Although \eqref{problem1} is a static problem, an extension to continuous time, as in \cite{mirrlees2013strategies}, is possible.} 

\begin{remark}
We emphasize that the choice of contract space $X$  requires minimal assumptions and that the {\sl weak compactness} requires no further restrictions on the set of possible contracts. By contrast, by the Kolmogorov-Riesz compactness theorem {\sl norm compactness} of the contract space would require the set of admissible transfers to be uniformly equicontinuous on average.  
\end{remark}

\subsubsection{The contracting problem over utility units}

Without continuity assumptions on the agent's utility function and compactness conditions on the contract space, {it is difficult to establish the existence of a solution for} the contracting problem over contingent claims.  To overcome this problem, we follow an approach that goes back at least to \cite{spear1987repeated} and \cite{schattler1993first} and that has more recently been used by \cite{sannikov2008continuous} and many others, and consider instead the following set
 $$C:=\{u(e_a+x):x\in X\}$$
of the agent's utility levels as the new contract space. Since $u$ is continuous and concave, $C$ is  $\sigma(L^1(Q),L^\infty(Q))$-compact.  This leads to the following equivalent  problem  over a compact contract space: find a mechanism $c: {\cal Q} \to C$ that maximizes 
  \begin{equation}\label{problem2}
  \begin{split}
 & \quad \inf_{\kappa \in \mathbb{K}}  \left\{ \int_{{\cal Q}}  E_{P'}\left[ v\left(e_p+e_a-u^{-1}(\,c_P\,)\right) \right]\textnormal{d}\kappa(P) + \alpha(\kappa) \right\} \\ 
 \text{subject to}
 & \quad {P\mapsto   E_{P'}\left[ v\left(e_p+e_a-u^{-1}(\,c_P\,)\right) \right]\quad  \text{is measurable} },
\\
& \quad E_{P}[c_P- { u( e_a)}]\geq 0,\quad \quad\quad\quad\qquad\quad P \in {\cal Q} \quad
\\
&\quad E_{P}[c_P- c_{\hat{P}}]\geq 0, \qquad\quad\quad\quad\quad\quad\quad  P,\hat P \in {\cal Q} .
 \end{split}
 \end{equation}
Although the change of variables helps to overcome the continuity problem  (at the level of utility units, the agent's utility functional is linear), a new problem emerges: space $L^\infty(Q)$ is essentially never separable, and hence, set $C$ cannot be expected to be metrizable when equipped with the weak topology\footnote{{The standard way to metrize the weak topology on $C$ is by considering a sequence $(\phi_n)\subset L^\infty(Q)$ and defining e.g. $C\times C\ni (f,g)\mapsto d(f,g):=\sum_n\frac{|\int\phi_n(f-g)dQ|}{2^n\|\phi_n\|_\infty}$. However, unless $(\phi_n)$ is dense, this is a pseudo-metric which is not a metric ($d(f,g)=0\not\Rightarrow f=g$). If $Q$ is not purely atomic, then no dense sequence in $L^\infty(Q)$ exists. }}.
We show in Section \ref{sec:applications} how our {general existence  result from Section \ref{abstract result menu} can be used to overcome this problem, and to establish the existence of an optimal risk-sharing rule under the preceding assumptions on utility functions $u$ and $v$ if all the densities in $\cal Q$ are uniformly $L^\infty(Q)$ bounded, that is, uniformly norm bounded.

\begin{remark}
If the initial endowments must belong to $L^2(Q)$, then $X$ and hence $C$ would be $\sigma(L^2(Q),L^2(Q))$-compact. For instance, this excludes any $\alpha$-Pareto distribution of wealth with $ \alpha\leq 2$.  Unlike the $\sigma(L^1(Q),L^\infty(Q))$-topology the $\sigma(L^2(Q),L^2(Q))$-topology is metrizable on bounded sets if $L^2(Q)$ is separable. This illustrates again the {economic rational} for working with non-metrizable contract spaces: Imposing minimal assumptions on the initial endowments naturally leads to $L^1(Q)$ - and hence nonmetrizable - contract spaces. The contract space is metrizable only under additional restrictions on the agent's initial endowments or utility function.       
\end{remark}  

\subsection{A reinsurance model with utilities on the whole real line} 
\label{sec:reinsurance_wholeline}

Let us now consider a modification of the preceding reinsurance model. We assume again that risk transfer occurs over a set of random variables ${\cal X}$. The agent is endowed with a \textsl{bounded} initial claim $e_a$, his utility function $u:\mathbb R_+ \to \mathbb{R}$ is still defined on the positive half-line but the principal's utility function  $v:\mathbb R \to \mathbb{R}$ is now defined on the whole real line, such as the exponential utility function. {Thus, in { contrast to} Subsection \ref{sec:reinsurance_halfline}}, the principal's payoff is not subject to a non-negativity constraint. In this case, we  {choose} as our set of admissible transfers a set 
\begin{equation} \label{set-X}
	{X} \subset \left \{  x \in {\cal X} : -e_a \leq x \right\} 
\end{equation}
that is closed in $L^0(Q)$, convex and {bounded in $Q$-probability}. For instance, if ${\cal X}$ is convex we may consider all transfers that are bounded below by $-e_a$ and bounded above by some positive random variable $y \in L^0(Q)$. Unlike in the previous subsection, set $X$ is not necessarily weakly compact. Nonetheless, we prove in Subsection \ref{sec:appwhole} that set  $C:=\{u(e_a+x):x\in X\}$ of the agent's utility levels is again $\sigma(L^1({Q}),L^\infty({Q}))$-compact under the above assumptions on $X$ if $u$ satisfies a certain asymptotic elasticity condition. Moreover, we prove that the contracting problem \eqref{problem2} has a solution if the agent's maximal attainable utility from reinsurance is finite and the type space ${\cal Q}$ is $L^\infty(Q)$-norm compact. 

\begin{remark}
If the principal's utility function $v$ is bounded as it was assumed in the preceding subsection, then the type space only needs to be norm bounded.  In the present setup, we can no longer  assume  that $v$ is bounded because it is now defined on the whole real line. \end{remark} 

\subsection{A hedging problem {with financial markets}}
\label{sec:reinsurance_wholeline_continuous}

{We close this section with a financial market framework for the general setting of Section \ref{sec:reinsurance_wholeline}.  Specifically, we consider an incomplete continuous-time financial market for which the asset price dynamics is described by a $d$-dimensional semi-martingale $(S_t)$ defined on a filtered probability space $(\Omega, ({\cal F}_t), \mathcal{F},Q)$. We consider a contracting problem between an investor (the principal) and a broker (the agent). The broker's assessment of the financial market dynamics is private knowledge. We assume that the principal cannot (or does not want to) trade in the financial market, either because she has no access to the market or because she finds it too costly. Instead, she may hedge her risk by exchanging payoffs with the agent. The agent is endowed with a bounded claim $e_a$ and accepts all transfers that he could super-replicate at zero initial investment by trading in the financial market over the time interval $[0,T]$. Here, a contingent claim $e \in L^0(Q)$ is called super{-replicable} at zero initial investment if an admissible trading strategy $\pi$ exists such that the resulting gains from trading satisfy $e \leq (\pi \cdot S)_T := \int_{0}^T \pi_t \textnormal{d}S_t$. An admissible  trading strategy  is an $({\cal F}_t)$ adapted $d$-dimensional stochastic process $(\pi_s)$ such that $(\pi \cdot S)_T$ is uniformly bounded from below. Let ${\cal X}$ be the set of super-replicable claims. We prove in Subsection \ref{sec:appwholecont} that the resulting set $X$ in (\ref{set-X}) is indeed closed in $L^0(Q)$, convex and {bounded in $Q$-probability.



\section{An Abstract Adverse Selection Problem} \label{abstract result menu}
Let $(\Theta,\tau^\Theta)$ and $(C,\tau^C)$ be topological spaces. We endow each of them with their Borel sigma-algebras, denoted $\B(\Theta) $ and $\B(C)$ respectively. 
Space $C$ is the \emph{space of contracts} that the principal  may offer the agent.  Space $\Theta$ is the \emph{space of possible types} that the agent may have.

\subsection{Menus}

If the agent is of type $\theta \in \Theta$ and contract $c \in C$ is chosen, the agent's utility is denoted by $U(\theta,c)$
and the principal's utility is denoted by $V(\theta,c)$.   Let  $$\mathbb{K}\neq\emptyset$$ be  a set of regular Borel probability measures on $\Theta$. Set  $\mathbb{K}$ comprises the principal's beliefs about the  agent's type distribution. 

\begin{assumption} \label{4a assumption 1a}
Throughout, we make the following standing assumptions:
\begin{enumerate}
\item[(a)]
$(C,\tau^C)$ and $(\Theta,\tau^\Theta)$ are compact Hausdorff spaces.
\item[(b)]
$U:\Theta \times C \to \RR$ is jointly continuous on $\Theta \times C$.
\item[(c)]
$V:\Theta \times C \to \RR$ is jointly upper semicontinuous on $\Theta \times C$. 
\end{enumerate}
\end{assumption}

\begin{remark}\label{r2}
The standard situation we focus on is one in which $C$ and $\Theta$ are subsets of Banach spaces $E$ and $E'$ that form a dual pair $\langle E,E' \rangle$, where $\tau^C$ is the weak topology on $C$ and $\tau^\Theta$ is a norm topology on $\Theta$, and where $U(\theta,c)$ is a continuous function of the duality product $\langle c,\theta \rangle$. { Since the dual of a Banach space is large if the space itself is small and vice versa, the duality of $E$ and $E'$ implies that a solvable adverse-selection model with {\sl large} type (contract) space typically requires a {\sl small} contract (type) space. In Section 4 we consider the case $E=L^1(Q)$ and $E'=L^\infty (Q)$ and consider a weakly compact subset $C \subset E$ as our contract space. As illustrated in the previous section this choice of contract space is natural when making the agent's utility the principal's decision variables.}
\end{remark}

The family of all $\tau^C$-closed subsets of $C$ (excluding $\emptyset$) is denoted by $\CL(C)$.
\begin{definition}
A  \textit{menu} is a non-empty closed subset of contracts. In other words, menus are the elements of $\CL(C)$. 
\end{definition}


A menu $D\in\CL(C)$ is automatically compact. Whenever $C$   is not metrizable, we cannot introduce the familiar Hausdorff distance on $\CL(C)$. Instead, we equip $\CL(C)$ with the Fell topology $\tau_F$. We recall the Fell topology and the related concept of Kuratowski-Painlev\'e convergence\footnote{
Limits in this sense are where the sets accumulate.
Note that we have to work with nets instead of sequences unless the contract set is metrizable. }
 in Appendix \ref{abstract result}. {The main benefit of working with the Fell topology is that the family of closed subsets of $C$ is compact as shown by the following lemma. The result is well-known; we give a proof in Appendix \ref{proofs} for the reader's convenience.} 
 
\begin{lemma}\label{lem Fell comp}
Space $(\CL(C),\tau_F)$ is compact Hausdorff.
\end{lemma}

{The use of the Fell topology is convenient if the contract space is non-metrizable, for instance if $C$ is a weakly compact subset of some $L^1$-space. Indeed, we prove that the agent's indirect utility function when presented a menu $D\in\CL(C)$ is continuous and that the corresponding argmax correspondence is upper semicontinuous. From this we conclude that the  principal's utility function as a function of types and menus is upper semicontinuous. Since the set of menus is compact in the Fell topology, this implies that an optimal menu exists. From this, we can then conclude the existence of an optimal contract mechanism.} 

\subsection{Agent}
Let us consider the  agent's utility over menus.  Given a menu $D\in \CL(C)$, an agent of type $\theta$ derives as utility the maximum $U^*(\theta,D)$ he can attain out of the contracts in the menu:
$$ U^*:\Theta \times \CL(C) \to \RR, \qquad U^*(\theta,D) \ := \ \sup_{d \in D} \, U(\theta,d). $$

\begin{proposition} \label{4a prop u star continuous}
 $U^*$ is continuous on $(\Theta \times \CL(C),\tau^\Theta\times\tau_F)$.
\end{proposition}

For the agent  let $\Phi$ be the optimal contracts, depending on types and menu:
\begin{equation} \label{eq corr phi}
\Phi:\Theta \times \CL(C) \ \twoheadrightarrow \ C, \qquad \Phi(\theta,D) \ := \argmax_{d \in D} \,U(\theta,d).
\end{equation}
Set $\Phi(\theta,D)$ contains all contracts within menu $D$, which are optimal for the agent of type $\theta$.

\begin{proposition} \label{4a prop correspondence}
$\Phi$ is a jointly upper hemicontinuous correspondence with compact non-empty values.
\end{proposition}

To formulate the individual rationality constraint, we assume that the agent agrees with a proposed menu only if he is left better off than if he had just retained his given (type-dependent) reservation utility $\underline{u}$, that is, if
\begin{equation}\label{eq:icc}
U^*(\theta,D) \ \geq \ \underline{u}(\theta), \qquad \text{for all } \theta \in \Theta.
\end{equation}
{Equation \eqref{eq:icc} imposes a constraint on the menu $D$, by requiring for each agent type the existence of a contract in $D$ which is better than a given outside option for an agent of that type.}
We assume that function $\underline{u}:\Theta \to \RR$ is common knowledge and the following standing assumption holds:\footnote{The assumption states that at least one individually rational contract exists. For instance, {one formally expects that 
the action ``$c^*=$do nothing'' (often $c^*=0$ is a concrete model) belongs to $C$ and  $\underline u(\theta):= U(\theta, c^*)$ holds for all $\theta\in \Theta$.} }

\begin{assumption} \label{4a assumption ir}
There exists $c^* \in C$ such that $
U(\theta,c^*) \geq \underline{u}(\theta),\text{ for all $\theta \in \Theta$}.$
\end{assumption}

To describe the set of individually rational menus, let ${T:\Theta \twoheadrightarrow \CL(C)}$
 be the correspondence defined by
\begin{eqnarray*}
 T(\theta)  &:= & \left\{D \in \CL(C) \, : \, U^*(\theta,D) \geq \underline{u}(\theta) \right\} \\
&=& U^*(\theta,\cdot)^{-1}[ \underline{u}(\theta),\infty), 
\end{eqnarray*}
which is closed-valued (by Proposition \ref{4a prop u star continuous}) and hence compact-valued. By Assumption \ref{4a assumption ir}, $c^* \in C$ exists  such that  $\{c^*\}\in T(\theta)$ for all $\theta\in \Theta$. As a result, we might define, with a slight abuse of notation,
\begin{equation} \label{eq def t}
T \ := \ \bigcap_{\theta \in \Theta} T(\theta) \ \subset \CL(C) ,
\end{equation}
which is compact and non-empty. We call $T$ the set of individually rational menus.

\subsection{Principal}

Given an agent's type $\theta \in \Theta$ and a menu $D$, the principal will only consider those contracts in $D$ that are optimal w.r.t.\ her utility for this given type.
This defines a type-dependent utility for the principal, namely:
$$ V^*:\Theta \times \CL(C) \to \RR, \qquad V^*(\theta,D) \ := \ \sup_{d \in \Phi(\theta,D)} V(\theta,d). $$
The counterpart to Proposition \ref{4a prop u star continuous} is the following:

\begin{proposition} \label{4a prop v usc}
$V^*$ is upper semicontinuous  on $(\Theta \times \CL(C),\tau^\Theta\times\tau_F)$.
\end{proposition}

To specify the principal's  utility, we recall that the (non-empty) set $\mathbb K$ consists of regular Borel probability measures on the type space and thus it collects the principal's possible priors  over the type of the agent. We fix a penalty function 
 \[
 	\alpha: \mathbb K \to \mathbb R
  \]
 and define  the principal's  utility of a menu $D$ in variational form as follows: 
\begin{equation}\label{eq def ppal util}
\inf_{\kappa\in\mathbb K} \left\{ \int_{\theta\in \Theta}V^*(\theta,D)\textnormal{d}\kappa(\theta) + \alpha(\kappa) \right\}.
\end{equation}
We stress that these integrals are well-defined if we assume that the type space is compact, as we have assumed so far, or if merely the principal's utility function is bounded.
The  principal's optimization problem over menus reads:
\begin{equation}\label{eq prob main delegated}
\sup_{D \in T} \, \inf_{\kappa\in\mathbb K}\, \left\{ \int_{\theta\in \Theta} V^*(\theta,D) \, \textnormal{d}\kappa(\theta) + \alpha(\kappa) \right\}
\end{equation}

The following result states a  general existence  of optimal contracts.
\begin{theorem} \label{4a thm existence optimal set}
Under the standing Assumptions \ref{4a assumption 1a} and \ref{4a assumption ir}, an optimal solution $D^*\in T$ to problem \eqref{eq prob main delegated} exists.
\end{theorem}

{

We now present two extensions of the preceding theorem. The first deals with non-compact type spaces but assumes that the principal's utility is bounded.

\begin{proposition}\label{prop-tight}
Let  
 $V$ be uniformly bounded. Further suppose that Assumption \ref{4a assumption ir} holds, as well as Assumptions \ref{4a assumption 1a} where the compactness of the type space $(\Theta,\tau^\Theta)$ is dropped. Then, an optimal solution $D^*\in T$ to the problem \eqref{eq prob main delegated} exists.
\end{proposition}

The second extension applies to a situation in which the principal evaluates her performance according to a worst-case approach. In this case, the existence of an optimal menu can be shown under considerably weaker conditions. The proof is a straightforward modification of the proof of Theorem \ref{4a thm existence optimal set}.

\begin{proposition}\label{worst-case}
Assume that Assumption \ref{4a assumption ir} and the following conditions hold:
\begin{enumerate}
\item[(a')]
$(C,\tau^C)$ is a compact Hausdorff space.
\item[(b')]
$U:\Theta \times C \to \RR$ is continuous on $C$ for any fixed type $\theta$.
\item[(c')]
$V:\Theta \times C \to \RR$ is upper semicontinuous on $C$ for any fixed type $\theta$.
\end{enumerate}
Then, an optimal solution $D^* \in T$ to the problem of maximizing  over $T$ the functional $D \mapsto \inf_{\theta \in \Theta} V^*(\theta,D)$ exists. 
\end{proposition} 
}

\subsection{On optimal contract mechanisms}\label{sec indirect measurable}

A contract mechanism is a mapping $\varphi:\Theta\to C$  from  the type space  into the contract space. 
A contract mechanism is individually rational if 
\begin{equation} \label{4a eq individual rational}
U(\theta,\varphi(\theta)) \ \geq \ \underline{u}(\theta), \qquad \text{for all } \theta \in \Theta.
\end{equation}
and incentive compatible if 
\begin{equation} \label{4a eq incentive compatible}
U(\theta,\varphi(\theta)) \ \geq \ U(\theta,\varphi(\theta')), \quad \text{for all } \theta,\theta' \in \Theta.
\end{equation}

We denote by $M$ the set of all contract mechanisms that are individually rational and incentive compatible. In  defining a principal-agent problem over optimal mechanisms, rather than over contract menus, the first difficulty  is that the integral 
$$
\int_\Theta V(\theta,\varphi(\theta))\textnormal{d}\kappa(\theta),
$$ 
is  not  defined unless   $\theta\mapsto V(\theta,\varphi(\theta))$ is measurable. Thus, we introduce set 
$$
\hat M:=\{\varphi\in M: \,\,V(\cdot,\varphi(\cdot))\text{ is measurable} \}
$$
and  define the principal-agent problem over contract mechanisms as follows:
\begin{equation}\label{eq prop main centralized}
\sup_{\varphi\in\hat M}\,\inf_{\kappa\in\mathbb K}\, \left\{  \int_\Theta V(\theta,\varphi(\theta))\textnormal{d}\kappa(\theta) + \alpha(\kappa) \right\}.
\end{equation}

\begin{theorem} \label{4a thm existence optimal mech}
Under the {assumptions of either Theorem \ref{4a thm existence optimal set} or Proposition \ref{prop-tight}, }
 an optimal solution $\varphi^*\in \hat M$ for problem \eqref{eq prop main centralized} exists.
\end{theorem}

The proof of Theorem \ref{4a thm existence optimal mech}   relies  on  Theorem \ref{4a thm existence optimal set}/Proposition \ref{prop-tight}. In fact, this proof reveals that delegated contracting implemented via contract menus is equivalent to centralized contracting implemented via contract mechanisms. {That is, the principal achieves the same optimal utility level by either proposing a suitable contract to each type of agent, or by proposing a suitable set of contracts (a menu) and letting the agent decide for himself.}



\section{Applications}\label{sec:applications}
In this section, we consider several examples for which we can establish the existence of an optimal contract by using the results of Section \ref{abstract result menu}. We are mostly interested in the case in which {the type space $\Theta$ is a subset of $L^\infty(Q)$ and the contract space $C$ in utility units is a $\sigma(L^1(Q),L^\infty(Q))$-compact subset of $L^1(Q)$.} The agent's utility function then  depends on the duality product  
\begin{equation}\label{eq duality}
	\langle \cdot, \cdot \rangle : L^1(Q) \times L^\infty(Q) \to \mathbb R.
	\end{equation}
The duality product is {\sl sequentially} continuous when $L^\infty(Q)$ is equipped with the $\sigma(L^\infty(Q),L^1(Q))$-topology (see \cite[Theorem 9.37]{Aliprantis2006}) but may fail to be continuous in general. Therefore, to guarantee continuity of the duality product we need to consider the norm topology on $L^\infty(Q)$ and assume that $\Theta$ is 
 subset of $L^\infty(Q)$.
 
First, we revisit the motivating examples presented in Section \ref{sec:motiv_examples}. Subsequently, we consider two examples in which the agent can trade in a financial market after transacting with the principal.  

\subsection{The motivating examples revisited} \label{mot ex}


\subsubsection{The reinsurance model of Section \ref{sec:reinsurance_halfline} }

Since the agent's utility function is strictly increasing, the contracting problems \eqref{problem1} and \eqref{problem2} are equivalent. We apply our abstract existence results to the contracting problem over utility levels. In the notation of Section \ref{abstract result menu}, $\tau^C$ is the $\sigma(L^1(Q),L^{\infty}(Q))$-topology on the {contract space}
 \[
	C = \{  u(e_a+x) : x \in [-e_a, e_p] \}
\]
of utility levels, $\tau^\Theta$ is the $L^{\infty}(Q)$ norm-topology on the {type space} $\Theta = {\cal Q}$, whereas
 \begin{align*}
	U(P,c) =E_P[c] \quad \mbox{and} \quad V(P',c)=  E_{P'}[v(e_p+e_a-u^{-1}(c))].
 \end{align*}
The utility function $U$ is  jointly continuous for our choice of topologies. To establish the upper semicontinuity of the utility function $V$, we recall the following result; see Proposition 2.10 in \cite{barbu2012convexity}.
\begin{lemma}\label{lem}\label{lemma-wc}
Let $L$ be a normed vector space and let $f:L \rightarrow \mathbb{R}$ be concave. Then, norm upper-semicontinuity of $f$   is equivalent to the weak upper-semicontinuity of $f$.
\end{lemma}

Since $u^{-1}$ is convex by definition, $-u^{-1}(\cdot)$ is concave, and hence $V$ is concave because $v$ is concave. Since $V$ is also $L^1(Q)$-norm upper semicontinuous by Fatou's lemma, it follows that {$V$  is (jointly)} weakly upper semicontinuous. Moreover, if ${\cal Q} \subset L^\infty(Q)$ is norm-bounded then $V$ is uniformly bounded  because $v$ is bounded. Thus, if an individually rational contract exists, then the existence of a solution to the contracting problem follows from Theorem \ref{4a thm existence optimal mech}. {By reverting the change of variables from $X$ to $C$}, we obtain an optimal mechanism for problem \eqref{problem1}. 


 \subsubsection{The reinsurance model of Section \ref{sec:reinsurance_wholeline}}\label{sec:appwhole}

This section analyzes the model of Section \ref{sec:reinsurance_wholeline} in a broader context. Our analysis is inspired by \cite[Section 2]{Kramkov2003}. We assume  that the agent's utility function $u:\mathbb R_+ \to \mathbb R$ satisfies the following conditions. 

\begin{assumption}\label{assumption-u}
The utility function $u:\mathbb R_+ \to \mathbb R$  is lower bounded, strictly increasing, strictly concave, continuously differentiable,  satisfies the Inada condition and the following asymptotic elasticity (AE) condition:
\[
	\limsup_{z \to \infty} \frac{zu'(z)}{u(z)} < 1. 
\] 
\end{assumption}

The (AE) condition was first introduced by \cite{Kramkov1999} {to show}  that optimal solutions exist for a wide class of  utility maximization problems. For our purposes, this  yields suitable compactness of the set of utility levels. In economic terms, it is a condition on the ratio of the marginal utility $u'(z)$ and the average utility $u(z)/z$ for $z$ large. We refer to \cite[Definition 3.1]{Sch02} and thereafter, for an exhaustive discussion on this condition and its interpretation.

\begin{example}
The {CRRA class }  $u^\gamma(z)=\frac{z^\gamma}{\gamma}$ for $\gamma \in (0,1)$ satisfies Assumption \ref{assumption-u}. 
\end{example}

{To handle utilities $v$ on the whole real line}, let {$\mathcal X\subset L^0(Q)$ be the contract space in payoff units} that satisfies the following condition: 

\begin{assumption}\label{assumption-abs}
Set {$\mathcal X$} is convex  and closed in $ L^0(Q)$. Further, for any $\lambda\in\mathbb R$ the following set is bounded in $Q$-probability:
$$\{x\in{\mathcal X}:\, x\geq\lambda\}.$$
\end{assumption}

We assume that the agent is endowed with a bounded random variable $e_a \in L^\infty(Q)$ and introduce the  set of  feasible risk transfers 
\begin{eqnarray}\label{Xea}
	{\mathcal X}(e_a) & :=& \{x\in L^0(Q):0 \leq x \leq e_a+\bar x \text{ for some }\bar x\in\mathcal {X}\} \\
	& =& (e_a+{\mathcal X}-L^0_+(Q))\cap L^0_+(Q).\nonumber
\end{eqnarray}
{This set represents the set of risk transfers that are feasible to the agent.} It follows from Assumption \ref{assumption-abs} and the Komlos lemma that ${\mathcal X}(e_a)$ is convex, solid,\footnote{This means that if $x \in {\mathcal X}(e_a)$ and $y \leq x$, then $y \in {\mathcal X}(e_a)$.} and closed in $ L^0(Q)$. Next, we introduce the polar set of ${\mathcal X}(e_a)$, namely,
$${\mathcal X}(e_a)^0:=\{x^0\in L^0_+(Q):\,E_Q[x x^0]\leq 1 \text{ for all }x\in \mathcal {X}(e_a)\}$$
along with the convex conjugate $u^*$ of $u$:
$${u^*}(y):=\sup_{z\in\mathbb{R}}\{u(z)-zy\}.$$

\begin{lemma} \label{lemma-sigma-compact}
Under Assumptions \ref{assumption-u} and \ref{assumption-abs}, assume that  some $y>0$ and a random variable $x^0\in {\mathcal X}(e_a)^0$ exist such that
\begin{align}\label{eq ass v Y}
	E_Q[{u^*}(x^0y)]<\infty. 
\end{align}
Then, the {contract space in utility units}
$
	C:=u(\mathcal {X}(e_a))=\{u(x) : x \in {\mathcal X} (e_a)\}
$
is contained in $L^1(Q)$ and is $\sigma(L^1(Q),L^\infty(Q))$-compact. 
\end{lemma}

The following lemma gives a sufficient condition that guarantees condition \eqref{eq ass v Y}. It essentially states that the agent's maximal attainable utility from risk transfer is finite. The proofs of Lemmas \ref{lemma-sigma-compact} and \ref{lemma-suff-cond} are presented in Appendix B. 

\begin{lemma}\label{lemma-suff-cond}
Let $\|e_a\|$ be considered as a constant function on $\Omega$. If
\begin{align}\label{eq ass finite u}
\sup_{x\in \mathcal {X}(\|e_a\|)} E_Q[u(x)]  <\infty,
\end{align}
where ${\mathcal  X}(\|e_a\|)$ defined in \eqref{Xea}, then condition \eqref{eq ass v Y} holds.
\end{lemma}

{With Lemma \ref{lemma-sigma-compact} and  \ref{lemma-suff-cond} at hand,} we can now use the same arguments as in the preceding subsection along with Theorem \ref{4a thm existence optimal mech} to establish the following result.  It shows that the contracting problem introduced in Section \ref{sec:reinsurance_wholeline} has a solution if the agent's utility from risk sharing is finite. 	

\begin{theorem}\label{Thm abstract KS}
 Under Assumptions \ref{assumption-u} and  \ref{assumption-abs}, if \eqref{eq ass finite u} holds
	and if ${\cal Q} \subset L^\infty(Q)$ is norm compact,\footnote{
Special cases of norm-compact subsets of $L^\infty(Q)$ can be constructed using Arzela and Ascoli's theorem when $\Omega$ is a topological space.} then the contracting problem of finding a mechanism $x:\mathcal{Q}\to \mathcal{X}(e_a)$ that  maximizes 
	{
  \begin{equation}\label{problem1concrete}
  \begin{split}
 & \quad \inf_{\kappa \in \mathbb{K}}  \left\{ \int_{{\cal Q}}  E_{P'}\left[ v\left(e_p-e_a+x_P)\right) \right]\textnormal{d} \kappa(P) + \alpha(\kappa) \right\} \\ 
 \text{subject to} 
 & \quad {P\mapsto  E_{P'}\left[ v\left(e_p-e_a+x_P)\right) \right],\quad ~ \text{is measurable} }
\\
& \quad E_{P}[u(x_P)-u(e_a)]\geq 0, \quad \quad\:\:   ~ \:\:\: P \in {\cal Q}
\\
&\quad E_{P}[u(x_P)- u(x_{\hat{P}})]\geq 0, \quad \quad \quad P,\hat P \in {\cal Q}, 
 \end{split}
 \end{equation}	
has a solution as soon as an individually rational contract exists (this is the case if, for instance, $0\in\X$). }
\end{theorem}

{Owing to Section \ref{abstract result menu}, the principal achieves in fact the same optimal utility level by either proposing a suitable contract to each type of agent (as in \eqref{problem1concrete}), or by proposing a suitable set of contracts (a menu) and letting agents decide for themselves. The same applies to the next application.}


\subsubsection{The hedging problem of Section \ref{sec:reinsurance_wholeline_continuous} }\label{sec:appwholecont}

The analysis of the hedging problem is related to the analysis in \cite{cvitanic2001utility}. A trading strategy $\pi$ is called \textit{admissible}, if the {gain from trade, modelled as a  stochastic integral $(\pi.S)_T=\int_0^T \pi_t \textnormal{d}S_t$,} is well--defined and lower bounded. We put
$${\mathcal X}_0:=\{x :x\leq (\pi.S)_T,\, \pi \text{ is admissible}\}.$$

{For the considered financial market, we require a form of no-arbitrage:}
\begin{assumption}\label{assumption-S}
A probability measure $Q'\approx Q$ exists such that the process $t\mapsto (\pi.S)_t$ is a $Q'$-local martingale for every admissible strategy $\pi$. Moreover, 
$$u_{max}:=\sup_{x\in {\mathcal X}_0} E_Q[u(x+e_a)]< \infty.$$

\end{assumption}

The key observation, made in \cite[Lemma 1]{Kramkov2003} for the case without random endowments, but easily obtainable in our setting as well, is the following:

\begin{lemma}
 Assumptions \ref{assumption-u} and \ref{assumption-S} imply   $\sigma(L^1(Q),L^\infty(Q))$-compactness   of 
$
	C=\{u(x) : x \in ({ \cal X}_0+e_a)\cap L_+^0(Q)\}.
$
 
\end{lemma}
 
Again, in the notation of Section \ref{abstract result menu}, $\tau^C$ is the $\sigma(L^1(Q),L^{\infty}(Q))$-topology on $C$ of utility levels, $\tau^\Theta$ is the $L^{\infty}(Q)$ norm-topology on set $\Theta = {\cal Q}$, and 
 \begin{align*}
	U(P,c) =E_P[c] \quad \mbox{and} \quad V(P',c)= E_{P'}[v(e+e_a-u^{-1}(c))].
 \end{align*}
The agent's utility   $U$ is  jointly continuous and the principal's utility  $V$ is upper semicontinuous. As a result, the hedging problem introduced in Section \ref{sec:reinsurance_wholeline_continuous} has a solution if the set of densities in ${\cal Q}$ is $L^\infty(Q)$-norm-compact. {This result is formally covered by Theorem \ref{Thm abstract KS}, but not quite, since 	hypothesis $u_{max}<\infty$ is weaker than \eqref{eq ass finite u}.}


\subsection{ Optimized risk sharing under incomplete markets}\label{exp uti}

In the previous examples, the agent had to decide whether to accept (or reject) a contract  based on his portfolio after the risk transfer. In this section we consider two examples in which he decides to accept (or reject) a contract based on the ``indirect utility'' that arises after investing in a financial market about which he has private information. Specifically, we consider a continuous time financial market model with one risk-free and one risky asset. The price of the risk-free asset is normalized to one. The discounted price of the risky asset follows a geometric Brownian motion with drift:
\[
	\frac{\textnormal{d}S_t}{S_t} = \textnormal{d}W_t + f'(W_t) \textnormal{d}t, \quad S_0 = s, \quad 0 \leq t \leq T
\]
where $W$ is a one-dimensional standard Brownian motion defined on the Wiener space $(\Omega, {\cal F}, ({\cal F}_t)_{0 \leq t \leq T}, Q)$ and $f \in \mathbb F$ where $\mathbb F$ is a class of real-valued twice continuously differentiable, uniformly bounded, uniformly equicontinuous functions $f: \mathbb{R} \to \mathbb{R}$ with common compact support that satisfy the normalization constraint $f(0)=0$\footnote{We think of $f'$ as being a constant outside a large compact set. In this case, the price dynamics essentially reduces to geometric Brownian motion.}. The closure $\bar{\mathbb F}$ of $\mathbb {F}$ w.r.t.\ the topology of uniform convergence is a compact subset of the set of continuous functions on $\mathbb{R}$ w.r.t.\ the supremum norm by the theorem of Arzela and Ascoli. By Girsanov's theorem, for any $f \in \bar{\mathbb F}$ there exists a unique equivalent martingale measure $P_f$ whose density is given by 
\[
	\frac{\textnormal{d}P_f}{\textnormal{d}Q} = e^{f(W_T)}.
\]
We assume that the pricing kernel is private knowledge to the agent and choose as our type set the function space 
\[
	\Theta := \bar{\mathbb F}
\]
equipped with the topology of uniform convergence.
In particular, the financial market is complete from the agent's point of view. The agent is endowed with a bounded claim $e_a \in {L^\infty}(Q)$. His preferences over payoffs are defined by an expected utility functional of the form $E_Q[u(\cdot)]$ for some Bernoulli utility function $u: \mathbb{R} \to \mathbb{R}$ that we choose to be the exponential (Subsection 4.2.1), respectively the logarithmic (Subsection 4.2.2) one.  
The principal is endowed with a bounded claim $e_p \in L^\infty(Q)$; her preferences over payoffs are defined by an expected utility functional of the form $E_Q[v(\cdot)]$ for some concave Bernoulli utility function $v:\mathbb{R} \to\mathbb{R}$. An admissible trading strategy for the agent is an adapted real-valued stochastic process vector $\xi$ that satisfies $E_Q[\int_0^T \xi_t^2 \textnormal{d}t] < \infty$.

\subsubsection{Risk-sharing with hedging}

Let us first consider a risk-sharing problem in which the agent can hedge his risk by trading in the financial market after transacting with the principal. We assume that 
\[
	u(y)=1-e^{-\alpha y}
\] 
with CARA parameter $\alpha > 0$.

{ The agent as an expert has better information about the financial market and  has a single pricing measure $P_f$.} Given his initial endowment $e_a \in L^\infty(Q)$, the budget set of  agent  $f \in \bar{\mathbb{F}}$ is given by
\begin{equation} \label{budget}
	{\cal B}(f) := \left\{ x \in L^1(Q) : E_f[x-e_a] \leq 0 \right\}
\end{equation}
where $E_f$ denotes the expectation under the pricing measure $P_f$. We notice that $E_f[e_a]$ is finite because $P_f$ has a bounded density w.r.t.~$Q$. The following result can be inferred from \cite[Chapter 3, Example 3.37]{FS}. It yields the optimal claim and the optimal utility from trading for any agent type. 

\begin{lemma}\label{x*}
The optimal attainable payoff over all admissible trading strategies for an agent of type $f$ in the budget set \eqref{budget} is  
\[
	x^* = -\frac{1}{\alpha} \ln \left(\frac{\textnormal{d}P_f}{\textnormal{d}Q} \right)+ E_f[e_a] + \frac{1}{\alpha}H(P_f|Q),
\] 
and the optimal utility from trading in the market is  
\begin{eqnarray*}\label{example2}
	E_Q[u(x^*)] = 1- e^{  - \alpha E_f[e_a]-H(P_f|Q)  }. 
\end{eqnarray*} 
Here, $H(P_f|Q):=E_f [\ln(\frac{\textnormal{d}P_f}{\textnormal{d}Q}) ] = E_Q\left[ e^{f(W_T)}f(W_T)  \right]<\infty$ denotes the relative entropy of $P_f$ with respect to  the Wiener measure $Q$.
\end{lemma}

{The density $\frac{dP_f}{dQ}$ corresponds to the private state price density of agent $f$. Thus, Lemma \ref{x*} clarifies  how the optimal payoff depends on the distance between $P_f$ and $Q$. } 

Let ${\cal X} \subset L^1(Q)$ be a closed set of uniformly integrable financial positions $x: \Omega \to \mathbb{R}$. 
By the Dunford-Pettis theorem, this is equivalent to ${\cal X}$ being $\sigma(L^1(Q),L^\infty(Q))$-compact. We take ${\cal X}$, equipped with the $\sigma(L^1(Q),L^\infty(Q))$ topology as the contract space: 
\[
	C = {\cal X}.
\]
When offered a payoff $x \in {\cal X}$ the agent's income before trading is $e_a + x$. It follows from the above lemma that his optimal optimal utility after trading in the financial market, that is, his indirect utility function is given by 
\[ 
	U\left( f,x \right) = 1- e^{  - \alpha E_f[e_a+x]-H(P_f|Q)  }.
\]
Since all functions $f \in \bar{\mathbb F}$ are uniformly bounded it follows from the dominated convergence theorem that the mapping $f \mapsto H(P_f|Q)$ is continuous. As a result, 
the mapping 
\[
	\left( f,x \right) \mapsto U\left( f, x \right)
\] 
on $\Theta \times {\cal X}$ is jointly continuous. A direct computation shows that a mechanism $f \mapsto x_{f}$ is incentive compatible if and only if
\[
	E_f[x_{f}- x_{f'}] \geq 0,\quad \mbox{ for all } f,f'\in\bar{\mathbb F}. 
\]
We assume that at least one {individually rational} incentive compatible mechanism exists. The principal's utility from transacting with  agent  $f$ is
\[
	V(x_f) = E_Q[v(e_P-x_f)].
\] 
Since $v$ is concave, $V$ is upper semicontinuous, and thus, it follows from Theorem \ref{4a thm existence optimal mech} that the principal's optimization problem  to  find a mechanism $x: \bar{\mathbb{F}} \to C$ that maximizes
\[
 \inf_{\kappa\in\mathbb K}\, \left\{ \int_{  \bar{\mathbb{F}}} V(x_f) \, \textnormal{d}\kappa(f) + \alpha(\kappa) \right\}
\]
subject to the usual incentive compatibility, individual rationality, and measurability condition, has a solution. {Furthermore, owing to Section \ref{abstract result menu}, the principal achieves in fact the same optimal utility level by either proposing a suitable contract to each type of agent (where indirect utilities are specified), or by proposing a suitable set of contracts (a menu) and letting agents decide for themselves according to their indirect utilities. The same applies to the next application.}

 \subsubsection{A model of optimal portfolio delegation}  

In the previous applications, the principal's utility function was independent of the agent type $f \in \bar {\mathbb{F}}$. In this section, we consider a simple model of optimal portfolio delegation in which the principal's utility function depends on the agent type. We now assume that the agent's Bernoulli utility function is 
\[
	u(y)=\ln(y).
\] 
We retain the assumption that the drift of the stock price process is private information to the agent. The budget set of an agent of type $f$ is the same as in (\ref{budget}). The following result can again be inferred from \cite[Chapter 3, Example 3.43]{FS}:

\begin{lemma}\label{x**}
The optimal attainable payoff over all admissible trading strategies for an agent of type $f$ in the budget set \eqref{budget} is  
$$
	x^* = E_f[e_a]\frac{\textnormal{d}Q}{\textnormal{d}P_f},
$$ 
and the optimal utility from trading in the market is  
\begin{eqnarray*}\label{example2}
	E_Q[u(x^*)] = \ln E_f[e_a] + H(Q|P_f) = \ln E_f[e_a] - E_Q[f(W_T)].
\end{eqnarray*} 
\end{lemma}

{The principal, considered an investor,  outsources} her portfolio selection to a manager (the agent) who has private information about the financial market dynamics and whose investment decisions in the above-specified financial market  {cannot be monitored}. Following \cite{OY} and \cite{backhoff2016conditional}, we restrict the class of admissible contracts {to be of an affine linear form}: the principal can offer the agent some  {state dependent payoff $x$ plus a $\beta$-fraction} of the gains or losses from trading. In particular, the agent is liable for possible losses.  {Formally,} a contract consists of some ${\cal F}$-measurable random variable $x$,  that we again assume to belong to some $\sigma(L^1(Q),L^\infty(Q))$-compact set ${\cal X}$ and some number $\beta \in [0,1]$, the proportion of wealth the agent can keep. The contract space
\[
	C = {\cal X} \times [0,1],
\]
is equipped with the product topology. When offered a contract $(x,\beta)\in C$, the agent's income from an admissible trading strategy $\xi$ is 
\[
	e_a + x + \beta \int_0^T \xi_t \textnormal{d} S_t.
\]
By Lemma \ref{x**}, since asset prices are martingales under measure $P_f$ the agent's optimal utility after trading in the market can be described by the continuous indirect utility function
\[
	U(f,x) = \ln E_f[e_a+x] + H(Q|P_f).
\]

{
In particular, the agent's optimal utility is independent of the performance part  $\int \xi dS$ of his contract\footnote{This has previously been observed in \cite{OY} and \cite{backhoff2016conditional}.}. Hence, and as in Subsection 4.2.1,  a mechanism $f \mapsto \varphi(f)=(x_{f},\beta_f)$ is  incentive compatible if and only if $	E_f[x_{f}-x_{f'}]\geq 0 $ for all $ f,f'\in\bar{\mathbb F}$. }

Although the agent's optimal utility is independent of the performance part of his contract, the optimal trading strategy $\xi^*$ and thus  the income $w^*$ to the principal  depend on $f$. Since the optimal claim for the agent can be decomposed as   $	x^* =  x + e_a + \beta \int_0^T \xi^*_t \textnormal{d} S_t$ 
the optimal gains from trading are given by 
\[
	\int_0^T \xi^*_t \textnormal{d}S_t = \frac{1}{\beta} \left( E_f[e_a+x] \frac{\textnormal{d}Q}{\textnormal{d}P_f} -x-e_a \right)
\]
and hence, the principal's income from {outsourcing the portfolio selection} is given by  
\[
	w^* = \frac{1-\beta}{\beta}\left( E_f[e_a+x] \frac{\textnormal{d}Q}{\textnormal{d}P_f} - (e_a+x) \right).
\]
Hence, unlike in the previous section, the principal's utility function now depends on the type of the agent with whom she interacts. Her utility when interacting with an agent of type $f$  and offering a contract $(x,\beta)\in C$ is given by
\[
	V\left(f, (x,\beta) \right) = E_{Q}\left[ v\left( e_p - x +  \frac{1-\beta}{\beta}\left( E_f[e_a+x] \frac{\textnormal{d}Q}{\textnormal{d}P_f} - (e_a+x) \right) \right)\right]. 
\]
This function is jointly upper semicontinuous. We can now use the same arguments as in Subsection 4.2.1 to conclude   that the principal's optimization problem  of finding a mechanism  $(x,\beta): \bar{\mathbb{F}} \to C$ that maximizes
\[
 \inf_{\kappa\in\mathbb K}\, \left\{ \int_{\bar{\mathbb{F}}} V(f,(x_f,\beta_f)) \, \textnormal{d}\kappa(f) + \alpha(\kappa) \right\}
\]
subject to the   incentive compatibility,  individual rationality, and measurability condition, has a solution (if a feasible mechanism exists) {by Theorem \ref{4a thm existence optimal mech}.}


\section{ Concluding Remarks}\label{concl}

We considered a very general setting for a mechanism design problem in the presence of adverse selection. Under mild hypotheses we proved that optimal contracts exist and that centralized contracting implemented via contract mechanisms is equivalent to delegated contracting implemented via contract menus. The guiding principle of our work was to use the utility levels of the agent as the relevant contract variables. When doing so, the agent's utility function is given by a bilinear form. This does not only significantly simplify the incentive compatibility condition but also yields a natural duality between the type and the contract type. Simultaneously, it naturally results in a framework in which the relevant topological spaces for the contracting problem may lack metrizability (and possibly separability too), as we illustrated with various examples. {It would be interesting to derive further insights into the structure of optimal contracts and to quantify the impact of robustness or ambiguity on optimal contracts as in \cite{auster2018robust} or \cite{song2018efficient}. In view of the many facets ambiguity adds to optimal contracting problems, a formal analysis of the characteristics of equilibrium contracts in our general environment does not seem appropriate. Conditions for pooling or separating equilibrium 
are relevant for many applications but these questions are best analyzed in a less general framework or even on a case-by-case basis.} 

{ We conclude this paper with a couple of technical comments.} Our study was greatly influenced by the works \cite{Page1992,Page1997}, and in particular, by Page's idea of considering contract menus rather than mechanisms. However, the lack of metrizability precludes us from applying the results therein, because it rules out the use of measurable selection results. For instance, in \cite{Page1992}, the contract space is given by sequentially compact subsets of $L^\infty$. Under an additional separability assumption that would not be needed in our setting, the contract space is metrizable. Compared with \cite{Page1992}, we reverse the role of the contract and type space: in all our examples, the contract spaces are subsets of $L^1$, which is very natural when the agent is an expected utility maximizer, whereas the type spaces are subsets of $L^\infty$. The choice of norm-bounded, respectively compact, type spaces in $L^\infty$ is an immediate consequence of the fact that the agent's utility function is given by a bilinear form after the transformation of variables.    

To the best of our knowledge, the existence of mechanisms does not follow from existing Komlos-type results used or established in other studies, such as \cite{Page1991,Balder1996,balder1990new,BalderHess1996}. Our setting does not seem to fulfill the hypotheses needed for such arguments. Under various topological assumptions on the contract space, it has been shown that if the principal knows the distribution $\mu$ of the agent types and if the set of incentive compatible mechanisms is convex,\footnote{This is the case if the contract space is convex and the agent's utility function is affine. The latter is guaranteed if mixed contracts are considered.} then a sequence $\{\varphi_k\}$ of strongly measurable incentive compatible mechanisms admits an almost surely Cesaro convergent subsequence, that is, a subsequence $\{n_k\}$ exists such that 
\[
	\frac{1}{n_k} \sum_{i=1}^{n_k}\varphi_i(\theta) \to \varphi^*(\theta) \quad \mbox{$\mu$-a.s. as } k \to \infty.
\]   
Our setting does not seem to fulfill the hypotheses needed for such arguments. More strikingly, our framework allows to consider very general sets of beliefs about the type distribution. By contrast, Komlos-type arguments \textsl{always} need a reference distribution w.r.t.\ which the principal's beliefs are absolutely continuous. Such a reference distribution does not exist if the principal evaluates her performance according to a worst-case approach w.r.t. the agent's type.  A further advantage of the approach used here is that we show that suitably-measurable contract mechanisms, and contract menus, are largely equivalent even though we cannot use measurable selection arguments. In other words, we do not require mechanisms to be measurable functions; only the principal's utility from implementing these mechanisms needs to be measurable.

\begin{appendix}

\section{  Abstract Results} \label{abstract result}

\subsection{The Fell Topology}
We denote by $C$ a compact Hausdorff topological space and recall that $\CL(C)$ stands for the family of all non-empty closed subsets of $C$. 
{
\begin{definition}\label{def Fell}
The {\em Fell topology} on $\CL(C)$, which we denote $\tau_F$, is the topology generated by the subbase consisting of all sets of the form
$$ V^- \ := \ \{A \in \CL(C) \, : \, A \cap  V\neq \emptyset \}, $$
and
$$W^+ \ := \ \{A \in \CL(C) \, : \, A \subset  W  \},$$
where $V$ and $W$ are non-empty open subsets of $C$.
\end{definition}

The definition of Fell topology when  $C$ is only Hausdorff can be found in \cite[Definition 5.1.1.]{be93}, and it reduces to Definition \ref{def Fell} in the present case of $C$ being compact.

Let us now recall the notions of lower and upper closed limits for nets of sets. For a net $\{A_\lambda\}_{\lambda \in \Lambda}$ in $\CL(C)$, let $\Li(A_\lambda)$ denote the set of all limit points of $\{A_\lambda\}_{\lambda \in \Lambda}$.
These are the points $c\in C$ such that each neighborhood of $c$ intersects $A_\lambda$ for all $\lambda$ in some residual subset of $\Lambda$. By contrast, $\Ls(A_\lambda)$ denotes the set of all cluster points of $\{A_\lambda\}_{\lambda \in \Lambda}$.
These are the points $c\in C$ such that each neighborhood of $c$ intersects $A_\lambda$ for all $\lambda$ in some cofinal subset of $\Lambda$. We always have $\Li(A_\lambda) \subset \Ls(A_\lambda)$. 

Since $C$ is a compact Hausdorff space, convergence w.r.t.\ the Fell topology and the notion of Kuratowski-Painlev\'{e} convergence of nets coincide \cite[Theorem 5.2.6]{be93}. 
Hence, we focus on the latter convergence, which is easier to apply in the proofs.

\begin{definition}
Let$\{A_\lambda\}_{\lambda \in \Lambda}$ be a net in $\CL(C)$ and $A \in \CL(C)$.
We say that $\{A_\lambda\}_{\lambda \in \Lambda}$ converges in the {\em Kuratowski-Painlev\'{e} sense} to $A$ (denoted $A=\, ^K\lim_{\lambda} A_\lambda$) if $ \Li(A_\lambda) \ = \ A \ = \ \Ls(A_\lambda)$.
\end{definition}

}

\subsection{Semicontinuity of an integral functional}

We made use of the following result, which is nontrivial in the present nonmetrizable setting:

\begin{lemma}
\label{lem Fat}
Let $X,Y$ be compact Hausdorff spaces, $\lambda$ a regular probability measure on $\B(X)$, and $g:X\times Y\to \mathbb{R}$ a jointly upper semicontinuous function. Then, 
$$Y \ni y\mapsto \int_X g(x,y)d\lambda(x),$$
is well-defined and upper semicontinuous.
\end{lemma}

\begin{proof}
Compact Hausdorff spaces are completely regular, and in a completely regular space every finite upper semicontinuous function equals the pointwise infimum of its continuous majorants:
$$g(x,y)=\inf_{c\in S}c(x,y), \quad S:=\{c :\, c\geq g \text{ everywhere, }c \text{  continuous} \};$$
see for example \cite[Proposition 7, No 7, Chapter IX.10]{Bourbaki58}.  Let us assume for the moment\footnote{Actually, this follows immediately from the fact that the topologies on $X$ and $Y$ are generated by uniformities: we prefer to give self-contained arguments, perhaps paving the way for future extensions.} that for $c:X\times Y\to \mathbb{R}$ continuous we have
\begin{align}\label{eq assume conti partial}
y_\alpha \to y \mbox{ implies } \sup_x|c(x,y_\alpha) - c(x,y)|\to 0.
\end{align}
From the duality of continuous functions and finite measures, we find that
$y\mapsto \int c(x,y)d\lambda(x),$
is continuous. Next, note that set $S$   is downwards directed, and hence, in particular, it can be seen as a decreasing net. In line with \cite[Proposition 2.13]{BWMajor}, where the regularity of $\lambda$ is needed, applied to the (upper-bounded) upper semicontinuous function $g$, we easily deduce
 $$\int_X g(x,y)d\lambda(x) =\int_X \left ( \inf_{c\in S} c(x,y)\right )d\lambda(x)= \inf_{c\in S}  \int_X  c(x,y)d\lambda(x).$$
Thus, the function on the l.h.s.\ is an infimum of continuous functions and therefore upper semicontinuous.
 
 Let us return to \eqref{eq assume conti partial}, which would be trivial in the metrizable setting. We fix $y$ and let $\epsilon>0$. For each $x$ we have by continuity the existence of neighborhoods $U_x^X$ and $V_x^Y$ of $x$ and $y$ respectively, such that
 $$|c(\bar{x},\bar{y}) - c(x,y)|\leq \epsilon/2 \mbox{ for all } \bar{x}\in U_x^X, \forall\bar{y}\in V_x^Y.$$
 We now cover $X\times \{y\}$, a compact, with $\bigcup_{x\in X}U_x^X\times V_x^Y$. Therefore, we obtain $x_1,\dots,x_n$ such that
 $$X\times \{y\} \subset\bigcup_{i\leq n} U_{x_i}^X\times V_{x_i}^Y,$$
 and we introduce $V:= \cap_{i\leq n} V^Y_{x_i}$. This is an open set containing $y$, and thus,
 $$X\times \{y\} \subset\bigcup_{i\leq n} U_{x_i}^X\times V.$$
 Therefore, given $\bar{x}$ arbitrary and $\bar{y}\in V$, there is some $i\leq n$ such that $(\bar{x},\bar{y})\in U_{x_i}^X\times V$ and thus
 \begin{eqnarray*}
 |c(\bar{x},y)- c(\bar{x},\bar{y})|&\leq&  |c(\bar{x},y)- c(x_i,y)|+  |c(x_i,y)- c(\bar{x},\bar{y})|\\
&\leq &\epsilon/2+\epsilon/2.
 \end{eqnarray*}
 This shows that  $\sup_{x\in X}|c(x,y)-c(x,\bar{y})|\leq \epsilon$ for all $\bar{y}\in V$, as desired.$\hfill \square$
\end{proof}

{
\begin{corollary}
\label{cor Fat}
Let $X$ be a Hausdorff space, and let  $Y$ be a compact Hausdorff space. Let  $\lambda$ a regular probability measure on $\B(X)$, and $g:X\times Y\to \mathbb{R}$ a jointly upper semicontinuous bounded function. Then 
$$Y \ni y\mapsto \int_X g(x,y)d\lambda(x),$$
is well-defined and upper semicontinuous.
\end{corollary}

\begin{proof}
Let $(X_n)$ be an increasing sequence of compact subsets of $X$ such that {$\lambda(X_n) \uparrow 1$, the existence of which is guaranteed by the (inner) regularity of $\lambda$}. By Lemma \ref{lem Fat}, the mappings
\[
	G_n(y) := \int_{X_n} g(x,y)d\lambda(x)
\]
are well-defined and upper semicontinuous for each $n \in \mathbb N$. {Subtracting from $g$ its lower bound, we may assume that $g$ is non-negative.} For any $y$ and $\lambda$ a.e.\ $x$ we have $1_{X_n}(x)g(x,y)\nearrow g(x,y)$, and hence, by sequential monotone convergence 
\[
	\lim_{n \to \infty} G_n(y)= \sup_n G_n(y) =  \int_{X} g(x,y)d\lambda(x)=:G(y) . 
\]  
Thus, $G$ is well-defined, and its upper semicontinuity remains to be proved. Since $g$ is bounded, we have $\| G_n - G\|_\infty \to 0$ as $n \to \infty$. Hence, if $y_\alpha\to y$ is a net, then for any $\epsilon>0$ and $n=n(\epsilon)$ big enough
$$\limsup G(y_\alpha)\leq\epsilon+ \limsup G_n(y_\alpha)\leq \epsilon+ G_n(y)\leq \epsilon+ G(y),$$
which concludes the proof.$\hfill \square$
\end{proof}

}

\section{Proofs} \label{proofs}

We now present the proof of Lemma \ref{lem Fell comp}, which actually follows by \cite[Proposition 5.1.2 and Exercise 5.1.4(a)]{be93}. For the reader's convenience, we provide the complete argument:

\begin{proof}{\it of Lemma \ref{lem Fell comp}}
Since $C$ is compact Haussdorf, for each $A_1,A_2\in \CL(C)$ disjoint, we can find $U_1,U_2\in \tau^C$ disjoint such that $A_i\subset U_i$. Then, $U_1^+$ and $U^+_2$ are disjoint open neighborhoods in $\tau_F$ of $A_1$ and $A_2$ respectively. Thus $\tau_F$ is Hausdorff.

To prove that $\tau_F$ is compact, we follow \cite[Theorem 5.1.3]{be93}. Let $\{V_\lambda:\lambda\in\Lambda\}$ and $\{W_\sigma:\sigma\in\Sigma\}$ be two families of non-empty open sets of $C$, such that 
$$
\CL(C)=\{ V^-_\lambda:\lambda\in\Lambda \}\cup \{W^+_\sigma:\sigma\in\Sigma\}.$$ 
By the Alexander subbase theorem, it suffices to check the existence of a finite subcovering for this kind of coverings, to obtain proof of compactness of $\tau_F$. Note that if $\Lambda$ is empty, then $C\in \cup_\sigma W^+_\sigma$, therefore, for some $\bar{\sigma}\in\Sigma$ we must have $W_{\bar{\sigma} }=C$ and consequently $W_{\bar{\sigma} }^+=\CL(C)$ is a finite subcovering. Hence, we now assume that $\Lambda$ is non-empty. If $\Sigma$ is empty, then $\forall c\in C:\{c\}\in \cup_\lambda V_\lambda^-$, and thus, these $V_\lambda$'s form an open covering of $C$. Then, we find $C= \cup_{i=1}^n V_{\lambda^n}$ and get 
$$
\CL(C)= \cup_{i=1}^n V_{\lambda^n}^- 
$$
 is a finite subcovering. 

Finally, if both $\Lambda$ and $\Sigma$ are non-empty, some $\sigma_0\in \Sigma$ must exist such that $(W_{\sigma_0})^c\subset \cup_\lambda V_\lambda$. Indeed, if this was not the case, we may choose $c_\sigma\in (W_{\sigma})^c\backslash \cup_\lambda V_\lambda$ for each $\sigma$, and then the closed set $\overline{\{c_\sigma:\sigma\in\Sigma\}}^\tau$ intersects no $V_\lambda$ and is not contained in any $W_\sigma$, contradicting the covering assumption. Since $(W_{\sigma_0})^c$ is compact, it follows $(W_{\sigma_0})^c\subset \cup_{k=1}^m V_{\lambda_k}$ and then clearly $W_{\sigma_0}^+\cup_k V_{\lambda_k}^-= \CL(C)$.$\hfill \square$
\end{proof}

\begin{proof}{\it of Proposition \ref{4a prop u star continuous}:}
Let $I: \Theta\times \CL(C) \twoheadrightarrow \Theta\times C$ be the \emph{identity} correspondence $$(\theta,D)\mapsto I(\theta,D):=\{\theta\}\times D.$$
It is easy to see that $I$ is a continuous correspondence (see \cite[Lemmata 17.4-17.5]{Aliprantis2006}) where the domain is given the topology $(\tau^\Theta\times\tau_F)$ and the range is given $(\tau^\Theta\times\tau^C)$. Since by Assumption \ref{4a assumption 1a}, function $U$ is jointly continuous, it follows by \cite[Berge Maximum Theorem 17.31]{Aliprantis2006} that $U^*$ is continuous too.$\hfill \square$
\end{proof}

\begin{proof}{\it of Proposition \ref{4a prop correspondence}:}
Since $D\in\CL(C)$ is compact, we have $\Phi(\theta,D)\neq\emptyset$ by the continuity of $U(\theta,\cdot)$. Further, since set $D\cap U(\theta,\cdot)^{-1}(\{U^*(\theta,D)\})\subset C$ is closed, it must be compact too, showing that $\Phi(\theta,D)$ is compact.

We now prove upper hemicontinuity. Since $C$ is compact, by \cite[17.11 Closed Graph Theorem]{Aliprantis2006} we need only check that the graph of $\Phi$ is closed.
Let $\{(\theta_\lambda,D_\lambda)\}_{\lambda\in\Lambda}$ be a net with $(\theta_\lambda,D_\lambda)\to(\theta,D)$, and let $f_\lambda\in \Phi(\theta_\lambda,D_\lambda)$ with $f_\lambda\to f$. We must show that  $f\in\Phi(\theta,D)$. Since $f$ is a limit point of $\{f_\lambda\}$, with $f_\lambda\in D_\lambda$ and $D=\Li(D_\lambda)$, we obtain $f\in D$. On the other hand, by definition
$$U(\theta_\lambda,f_\lambda)=U^*(\theta_\lambda,D_\lambda).$$
Since $U$ and $U^*$ are continuous (Proposition \ref{4a prop u star continuous}), we can go into the limit, finding
$$U(\theta,f)=U^*(\theta,D).$$
This finding and $f\in D$ conclude the proof.$\hfill \square$
\end{proof}

\begin{proof}{\it of Proposition \ref{4a prop v usc}:}
By Proposition \ref{4a prop correspondence}, the correspondence $(\theta,D) \mapsto \Phi(\theta,D)$ is upper hemicontinuous and has non-empty compact values. On the other hand, by Assumption \ref{4a assumption 1a}, the function $(\theta,f) \mapsto V(\theta,f)$ is upper semicontinuous. It follows that the mapping $ V^*(\cdot, \cdot)$ is upper semicontinuous; see \cite[Theorem 2 (p.\ 116)]{Berge1963} or \cite[Lemma 17.30]{Aliprantis2006}.$\hfill \square$
\end{proof}

\begin{proof}{\it of Theorem \ref{4a thm existence optimal set}:}
First, define $F^\kappa:\CL(C) \to \RR$ by $$F^\kappa(D) := \int_\Theta V^*(\theta,D) d\kappa(\theta),$$
for $\kappa \in \mathbb K $. We  have $ V^*(\theta,D)\leq \sup_{\theta\in \Theta,\,c\in C} V(\theta,c)$, the r.h.s.\ of which is finite by compactness and u.s.c.\ of  $V$. By Lemma \ref{lem Fat}, taking $g:=V^*$, we have that $F^\kappa$ is well-defined and, in fact, upper semicontinuous in $\CL(C)$. It follows that $\inf_{\kappa\in \mathbb K}F^\kappa(\cdot)$ is also upper semicontinuous. The set $T$ is non-empty and compact, see \eqref{eq def t}. The result follows.$\hfill \square$
\end{proof}

{
\begin{proof}{\it of Proposition \ref{prop-tight}:}
Since $V$ is assumed bounded, the proof follows from the same arguments as for Theorem \ref{4a thm existence optimal set}], using Corollary \ref{cor Fat} instead of Lemma \ref{lem Fat}. $\hfill \square$
\end{proof}
}

\begin{proof}{\it of Theorem \ref{4a thm existence optimal mech}:}
Let $D^*$ be the optimizer of \eqref{4a thm existence optimal set}. Since set $\Phi(\theta,D^*)$ is non-empty and compact for each $\theta$, we deduce that set 
$$\argmax_{d\in\Phi(\theta,D^*)}V(\theta,d)$$ is non-empty. By axiom of choice, we select $\varphi^*(\theta) \in\argmax_{d\in\Phi(\theta,D^*)}V(\theta,d),$ thus, by definition,
$$V(\theta,\varphi^*(\theta))=V^*(\theta,D^*),\text{ for all $\theta$}.$$
Observe that $V(\cdot,\varphi^*(\cdot))$ is measurable, because $V^*(\cdot,D^*)$ is measurable. In addition, since $\varphi^*(\theta)\in D^*$ for each $\theta$, we deduce in particular
$$
\varphi^*(\theta)\in\Phi(\theta,D^*),
$$
which yields that 
$$
U(\theta,\varphi^*(\theta))\geq U(\theta,\varphi^*(\theta ')), \mbox{ for all } \theta' \in \Theta.
$$
Finally, since $D^*\in T$ and $\varphi^*(\theta)\in\Phi(\theta,D^*)$ we conclude that $U(\theta,\varphi^*(\theta))\geq \underline{u}(\theta)$. Thus far, we have established that $\varphi^*\in\hat M$. Now, let $\varphi\in\hat M$ arbitrary, and let $D$ be the closure of $\varphi(\Theta)$. It follows that $\varphi(\theta) \in \Phi(\theta,D)$, since $\varphi$ is incentive compatible and by continuity of $U$. Thus $V(\theta,\varphi(\theta))\leq V^*(\theta,D)$ by definition. Since $\varphi$ is individually rational, it follows $D\in T$. All in all,
\begin{eqnarray*} 
\int_\Theta V(\theta,\varphi(\theta))d\kappa(\theta)&\leq &\int_\Theta V^*(\theta,D)d\kappa(\theta)\\
&\leq & \int_\Theta  V^*(\theta,D^*)d\kappa(\theta)=\int_\Theta  V(\theta,\varphi^*(\theta))d\kappa(\theta), 
\end{eqnarray*} 
from which $\varphi$ is optimal for Problem \eqref{eq prop main centralized}.$\hfill \square$

\end{proof}

\begin{proof}{\it of Lemma \ref{lemma-sigma-compact}:}
Without loss of generality, we may assume that $u\geq 0$. By \eqref{eq ass v Y} and the definitions of $v$ and ${\mathcal X}(e_a)^0$ we obtain that $C$ is bounded in $L^1(Q)$. By Eberlein-Smulian's theorem, it is enough to check sequential compactness. By contradiction suppose that $C$ is not $\sigma(L^1(Q),L^\infty(Q))$-relatively compact, and hence not uniformly integrable. As in the proof of \cite[Lemma 1]{Kramkov2003}, we obtain the existence of a sequence $\{g_n\}\subset{\mathcal X}(e_a)$, a sequence of disjoint measurable events $\{A_n\}$ and an $\alpha>0$ such that
$$E_Q[u(g_n)1_{A_n}]\geq \alpha.$$
Accordingly, if we define $H_k:=\sum_{n\leq k}u(g_n)1_{A_n}$ we find $E_Q[H_k]\geq k\alpha$. Note that as in \cite[Lemma 6.3]{Kramkov1999}, $u^*(y/2)\leq au^*(y)+b$, and thus, in particular,
$$\qquad u^*(y2^{-m})\leq a^m u^*(y)+(a^{m-1}+\dots+a^1+a^0)b.$$
Now, clearly 
$$H_k=u\left(\sum_{n\leq k}g_n1_{A_n}\right)\leq u^*(Y)+Y \sum_{n\leq k}g_n1_{A_n},$$
for every non-negative random variable $Y$. Let us temporarily assume that we could choose $Y^*$ independent of $k$ and such that for some $L<\alpha$:
\begin{align}\label{eq d 1}
E_Q\left[Y^*\sum_{n\leq k}g_n1_{A_n}\right]\leq Lk.
\end{align}
From the above computations we then have $k\alpha\leq E_Q[H_k]\leq E_Q[u^*(Y^*)]+Lk,$ 
and hence, if further 
\begin{align}\label{eq d 2}
E_Q[u^*(Y^*)]<\infty,
\end{align}
this would yield a contradiction with $L<\alpha$. 

We now show the existence of $Y^*$ fulfilling \eqref{eq d 1} and \eqref{eq d 2}. 
By Assumption \eqref{eq ass v Y}, there exists $y>0,Y\in {\mathcal X}(e_a)^0$ such that $E_Q[u^*(yY)]<\infty$.
 However, then $$E_Q[u^*(2^{-m}yY)]<\infty,$$
for all $m>0$. We now take $m$ large so that $L:=y2^{-m}<\alpha$ and define $Y^*:=LY$. Then
 $$E_Q\Big[Y^*\sum_{n\leq k}g_n1_{A_n}\Big]\leq E_Q\Big[Y^*\sum_{n\leq k}g_n\Big]\leq Lk.$$
 Hence $Y^*$ fulfills \eqref{eq d 1}-\eqref{eq d 2} as desired.
 
 It remains to show that $C$ is weakly closed. First we show that $C$ is convex. Since $u^{-1}$ is increasing and $u$ concave, for $X,\bar X\in {\mathcal X}(e_a)$ and $\beta\in[0,1]$ we have
 \begin{align*}
 	u^{-1}\left ( \beta u (X)+(1-\beta)u(\bar X) \right ) &\leq u^{-1}\left ( u(\beta X+(1-\beta)\bar X) \right )\\
	&=\beta X+(1-\beta)\bar X \,\in {\mathcal X}(e_a).
\end{align*}
Because ${\mathcal X}(e_a) $ is solid, then $u^{-1}\left ( \beta u (X)+(1-\beta)u(\bar X) \right )\in {\mathcal X}(e_a)$, and therefore $\beta u (X)+(1-\beta)u(\bar X)\in C$. Since $C$ is convex, it is weakly closed if and only if	 it is strongly closed. Let us recall 
the bipolar theorem of \cite[Theorem 1.3]{brannath1999bipolar}, which applies here since ${\mathcal X}(e_a)$ is convex, solid, and closed in $L^0(Q)$, that states
$$X\in {\mathcal X}(e_a) \iff  E_Q[XY]\leq 1 \mbox{ for all } Y\in  {\mathcal X}(e_a)^0.  $$
Now let $\{ U(X_n)\}_n \subset C$ converge in $L^1$-norm to $Z$, and let $Y\in  {\mathcal X}(e_a)^0$ be arbitrary. Then,
$$E_Q[Yu^{-1}\circ u(X_n)]=E_Q[YX_n]\leq 1,$$
and therefore, by Fatou's lemma $E_Q[Yu^{-1}(Z)]\leq 1$ too, and by the bipolar theorem $u^{-1}(Z)\in{\mathcal X}(e_a) $. Consequently, $Z\in C$, finishing the proof.$\hfill \square$
\end{proof}

\begin{proof}{\it of Lemma \ref{lemma-suff-cond}:}
Observe that ${\mathcal X}(e_a)\subset{\mathcal X}(\|e_a\|)$, and, as a consequence,
$$\forall X\in {\mathcal X}(\|e_a\|):\,E_Q[XY]\leq 1 \Rightarrow \forall X\in {\mathcal X}(e_a):\,E_Q[XY]\leq 1.$$
Assumption \eqref{eq ass finite u} allows us to apply \cite[Theorem 3.1]{Kramkov1999} to the pairing between ${\mathcal X}(\|e_a\|)$ and its polar. As a particular consequence, there exists $y\geq 0$ and $Y$ in the polar of ${\mathcal X}(\|e_a\|)$ such that $$E_Q[u^*(yY)]<\infty.$$
Since the polar of ${\mathcal X}(e_a)$ is larger, we conclude the proof. $\hfill \square$
\end{proof}

\end{appendix}

\begin{acknowledgements}
The authors would like to thank the anonymous referees and editors for their valuable advice.
\end{acknowledgements}

%

\bibliographystyle{plain}    


\bibliography{bib2019Juli}

\end{document}